\documentclass[a4paper,10pt,english]{article}

\usepackage{babel}
\usepackage{amsmath,amsfonts,amssymb,amsthm,yhmath,mathrsfs, amsbsy}
\usepackage{latexsym}
\usepackage{mathtools}
\usepackage[utf8]{inputenc}
\usepackage{fancyhdr}
\usepackage[nottoc]{tocbibind}
\usepackage{indentfirst}
\usepackage{paralist}
\usepackage{hyperref}
\usepackage{graphicx}
\usepackage{dcpic, pictex}
\usepackage{etoolbox}
\usepackage{tikz}
\usepackage{enumerate}

\usepackage{vmargin}
\setmarginsrb{13mm}{4mm}{13mm}{9mm}%
          {5mm}{5mm}{5mm}{7mm}

\newtheorem{teo}{Theorem}[section]
\newtheorem{pro}[teo]{Proposition}
\newtheorem{lemma}[teo]{Lemma}
\newtheorem{coro}[teo]{Corollary}

\newtheorem{con}[teo]{Conjecture}
\theoremstyle{definition}
\newtheorem{deff}[teo]{Definition}

\newtheorem{rem}[teo]{Remark}
\newtheorem{esem}[teo]{Example}

\newcommand{\rest}{\mathbin\restriction}
\newcommand{\N}{\mathbb N}
\newcommand{\Z}{\mathbb Z}
\newcommand{\Q}{\mathbb Q}
\newcommand{\R}{\mathbb R}
\newcommand{\Pf}{\mathcal P_{fin}}
\newcommand{\F}{\mathcal F}

\newcommand{\End}{\mathrm{End}}

\newcommand{\id}{\mathrm{id}}

\numberwithin{equation}{section}

\renewenvironment{thebibliography}[1]{
  \begin{oldthebibliography}{#1}
    \setlength{\itemsep}{0em}
    \setlength{\parskip}{0em}
}
{
  \end{oldthebibliography}
}

\title{Additivity of the algebraic entropy \\ for locally finite groups with permutable finite subgroups}

\author{Anna Giordano Bruno \and Flavio Salizzoni}

\date{\small{\texttt{anna.giordanobruno@uniud.it} \quad \texttt{flavio.salizzoni@gmail.com}\\ Universit\`a degli Studi di Udine\\ Dipartimento di Scienze Matematiche, Informatiche e Fisiche \\ Via delle Scienze 206, 33100 Udine (Italy)}}

\begin{document}

\maketitle

\begin{abstract}
The additivity with respect to exact sequences is notoriously a fundamental property of the algebraic entropy of group endomorphisms. It was proved for abelian groups by deeply exploiting their structure. On the other hand, a solvable counterexample was recently found, showing that it does not hold in general. Nevertheless, we give a rather short proof of the additivity of the algebraic entropy for locally finite groups that are either quasihamiltonian or $FC$-groups.
\end{abstract}

\section{Introduction}

In analogy with the measure entropy by Kolmogorov \cite{Kol} and Sinai \cite{Sinai}, Adler, Konheim and McAndrew \cite{AKM} investigated the topological entropy for continuous selfmaps of compact spaces. They also introduced the algebraic entropy for endomorphisms of discrete abelian groups, %in connection with the topological entropy via Pontryagin duality.
%
%Later on, Hood \cite{hood} extended Bowen-Dinaburg's entropy (see \cite{B,Din}) to uniformly continuous selfmaps of uniform spaces. This notion of entropy coincides with the topological entropy from \cite{AKM} in the compact case, and it can be considered in particular for continuous endomorphisms $\phi$ of locally compact groups $G$ (see \cite{GBV}).
%(when the given compact topological space is endowed with the unique uniformity compatible with the topology). For this reason we call \emph{topological entropy} also Hood's extension and we denote it by $h_{top}$ (see \S\ref{htop-sec} for a definition). 
that was gradually developed by Weiss \cite{W} and Peters  \cite{P1}. More recently, it was thoroughly described for torsion abelian groups by Dikranjan, Goldsmith, Salce and Zanardo \cite{DGSZ}, and for abelian groups in \cite{DGB}, where a connection was pointed out with Lehmer Problem from number theory. This is based on the so-called Algebraic Yuzvinski Formula: the algebraic entropy of an endomorphism $\phi$ of $\Q^k$ coincides with the Mahler measure of the characteristic polynomial of $\phi$ over $\Z$ (see \cite{GBV,GBV2} -- see \cite{LW,Y1} for the Yuzvinski Formula concerning the topological entropy).

In \cite{DGBpak} the definition of algebraic entropy was extended to all group endomorphisms and all basic properties were verified. Moreover, a close connection was pointed out between the algebraic entropy and the classical growth theory of finitely generated groups due to Milnor \cite{Mil}, Gromov \cite{Gro}, Grigorchuk \cite{Gri}, etc. (see also \cite{DGB3}). In \cite{DGB,GBSp,GBSp1} various results were given on the growth of group endomorphisms; at the moment, the main result in this context is the counterpart of Milnor-Wolf Theorem for endomorphisms of elementary amenable groups.

In another direction the algebraic entropy was recently extended to left actions of cancellative right amenable semigroups on abelian groups in \cite{DFG} (see also the work of Virili \cite{V3} with applications to Kaplansky's Stable Finiteness Conjecture and Zero-Divisors Conjecture).
Moreover, we refer the interested reader to \cite{GBST,P2,V1,XST} for results on the algebraic entropy for continuous endomorphisms of locally compact groups, and to \cite{DGB1,DGB2,GB,V2,W} for the connection of the algebraic entropy with the topological entropy via Pontryagin duality.

\medskip
%:
Now we give the definition of algebraic entropy in the setting we are interested in, that is, for group endomorphisms.
Let $G$ be a group and $\phi\in\End(G)$. For every non-empty subset $X$ of $G$ let $T_0(\phi,X)=\{1\}$ and, for all $n\in\mathbb{N}_+$, let
\begin{equation}
T_n(\phi,X)=X\phi(X)\phi^2(X)\cdots\phi^{n-1}(X)\,.
\end{equation}
%and the \emph{$\phi$-trajectory} of $X$ is
%\begin{equation}
%T(\phi,X)=\prod_{n\in\mathbb{N}}\phi^{n}(X)\,.
%\end{equation}
Denote by $\mathcal P(G)$ the power set of $G$ and let $$\Pf(G)=\{X\subseteq G\mid X\ \text{is finite and non-empty}\}.$$ 
If $X\in\Pf(G)$, then $T_n(\phi,X)\in\Pf(G)$ for every $n\in\N$, and if $1\in X$, we have an increasing chain, that is, $T_n(\phi,X)\subseteq T_{n+1}(\phi,X)$ for every $n\in\N$.

%\begin{rem}%\label{base}
%Let $G$  be a group, $X\in\Pf(G)$ and $\phi\in\End(G)$.
%\begin{enumerate}[(a)]
%\item If $X$ is finite, then the $n$-th $\phi$-trajectory of $X$ is finite.
%\item If $1\in X$, then $T_n(\phi,X)\subseteq T_{n+1}(\phi,X)$.
%\end{enumerate}
%\end{rem}

Denote by $$\ell:\mathcal P(G)\to \R_{\geq0}\cup\{\infty\}$$ the function defined by $\ell(X)=\log|X|$ for every $X\in\Pf(G)$ and $\ell(X)=\infty$ for every infinite $X\in\mathcal P(G)$.

If $X\in\Pf(G)$, then the limit
\begin{equation}\label{limite}
H(\phi,X)=\lim_{n\rightarrow\infty}\frac{\ell(T_n(\phi,X))}{n}
\end{equation}
exists and $H(\phi,X)=\inf_{n\in\N}\frac{\ell(T_n(\phi,X))}{n}$ (see \cite[Lemma 5.1.1]{DGBpak}); in particular, $H(\phi,X)$ is finite and it is called the \emph{algebraic entropy of $\phi$ with respect to $X$}.
The \emph{algebraic entropy} of $\phi$ is
$$h(\phi)=\sup\{H(\phi,X)\mid X\in\Pf(G)\}\,.$$

Consider the category of algebraic dynamical systems, that has as objects the pairs $(G,\phi)$ with $G$ a group and $\phi\in\End(G)$, and as morphisms between objects $(G,\phi)$ and $(H,\psi)$ the group homomorphisms $\xi:G\to H$ such that $\psi\xi=\xi\phi$. The algebraic entropy is an invariant of this category, meaning that isomorphic algebraic dynamical systems have the same algebraic entropy: if $G$ and $H$ are groups, $\phi\in\End(G)$ and $\psi\in\End(H)$, and there exists an isomorphism $\xi:G\to H$ such that $\psi=\xi\phi\xi^{-1}$, then $h(\phi)=h(\psi)$ (see \cite{DGBpak}).

\smallskip
A fundamental property of the algebraic entropy, as well as of other entropy functions in mathematics, is the so-called Addition Theorem:

\begin{deff}
We say that \emph{the Addition Theorem holds} for a group $G$, $\phi\in\End(G)$ and $H$ a $\phi$-invariant normal subgroup of $G$, and write that \emph{$AT(G,\phi,H)$ holds}, if 
$$h(\phi)=h(\phi\rest_H)+h(\bar \phi_{G/H})\,.$$
where $\bar\phi_{G/H}\in\End(G/H)$ is induced by $\phi$.

We say that \emph{the Addition Theorem holds for $G$}, and write that \emph{$AT(G)$ holds}, if $AT(G,\phi,H)$ holds for every $\phi\in\End(G)$ and every $\phi$-invariant normal subgroup $H$ of $G$.
\end{deff}
 
In \cite{DGSZ} the Addition Theorem was proved for torsion abelian groups, deeply exploiting the structure of such groups. This result was extended to every abelian group in \cite{DGB0}, again making a very heavy use of the properties of abelian groups in order to reduce to the case of finite dimensional rational vector spaces, where the Algebraic Yuzvinski Formula applies.
Recently, the Addition Theorem from \cite{DGSZ} was extended also to left actions of cancellative right amenable monoids on torsion abelian groups in \cite{DFG}.
Since an endomorphism of an abelian group $G$ induces on $G$ the structure of a $\Z[x]$-module, and viceversa, one can consider the algebraic entropy as an invariant of the category of $\Z[x]$-modules. So, the importance of the Addition Theorem comes also from the fact that, together with the property of being continuous with respect to direct limits, it implies that the algebraic entropy is a length function of the category of $\Z[x]$-modules in the sense of Northcott and Reufel \cite{NR} and V\'amos \cite{Vamos} (see~\cite{DGB0} for the details on this connection, and see also \cite{SVV}).

On the other hand, an example was given in \cite{GBSp} of a solvable group $G$ for which $AT(G)$ does not hold true. For example, take $G$ to be the Lamplighter group $G=\Z_2^{(\Z)}\rtimes \Z$, with $\phi=\id_G$ and $H=\Z_2^{(\Z)}$; then $h(\id_G)=\infty$, while $h(\id_H)=0=h(\id_{G/H})$. Indeed, the identity map of an abelian group has always zero algebraic entropy, while a (finitely generated) group $G$ has exponential growth precisely when $h(\id_G)=\infty$ (see~\cite{DGBpak,DGB0,GBSp} for more details).

Nevertheless, it would be relevant to understand for which non-abelian groups $G$ we have that $AT(G)$ holds. For example, the conjecture that $AT(G)$ holds for every nilpotent group $G$ is still open.

In this paper we are interested in the following conjecture from \cite{GBST,GBSp}, and we prove it in a particular case (see Theorem~\ref{atq}) covering recent results from \cite{GBST,XST} as well as the Addition Theorem for torsion abelian groups from \cite{DGSZ}.

\begin{con}%[Addition Theorem]
If $G$ is a locally finite group, then $AT(G)$ holds.%Let $G$ be a locally finite group, $\phi\in\End(G)$, $H\unlhd G$ $\phi$-invariant and $\bar \phi\in\End(G/H)$ the endomorphism induced by $\phi$. Then
%$$ h(\phi)=h(\phi\rest_H)+h(\bar \phi)\,.$$
\end{con}

First of all note that in the class of locally finite groups the computation of the algebraic entropy becomes more comfortable. %in the sense that we can use finite subgroups
Indeed, denoting by $$\mathcal{F}(G)=\{F\leq G\mid F\ \text{is finite}\}\subseteq \Pf(G)$$ the family of all finite subgroups of a group $G$, we have that $G$ is locally finite precisely when $\F(G)$ is cofinal in $\Pf(G)$ with respect to the order given by the inclusion.
So, for a locally finite group $G$ and $\phi\in\End(G)$, we have that (see Remark~\ref{rem1})
%\begin{rem}\label{rem1}
%Let $G$ be a group and $\phi\in\End(G)$.
%The function $H(\phi,-):\Pf(G)\to\R_{\geq0}$ is non-decreasing, that is, if $X,X'\in \Pf(G)$ and $X\subseteq X'$, then $H(\phi,X)\leq H(\phi,X')$. Therefore, if $\mathcal{F}$ is a cofinal subfamily of $\Pf(G)$ with respect to the inclusion, then
%$$h(\phi)=\text{sup}\{H(\phi,X)\mid X\in\mathcal{F}\}\,.$$
%\end{rem}
%
%Following the first definition of algebraic entropy from \cite{W}, we can define another \emph{algebraic entropy} of $\phi\in\End(G)$ as
%\begin{equation}
%\text{ent}(\phi)=\sup\{H(\phi,F)\mid F\in\mathcal{F}(G)\}\,.
%\end{equation}
%Clearly, $$\ent(\phi)\leq h(\phi).$$

%\begin{rem}
%If $G$ is locally finite (i.e., every finitely generated subgroup of $G$ is finite), equivalently, $\F(G)$ is cofinal in $\Pf(G)$.
%So, taking $\mathcal F=\F(G)$, we have that
$$h(\phi)=\sup\{H(\phi,F)\mid F\in\mathcal{F}(G)\}.$$
%\end{rem}
This means that we can use only the finite subgroups of $G$ to compute the algebraic entropy of $\phi$. 

In addition we will need also that each $T_n(\phi,F)$ is a subgroup of $G$ for every $F$ in a cofinal subfamily of $\F(G)$. For this reason we will introduce the groups in Definition~\ref{fqhdef}.

\smallskip
We start recalling a notion due to Ore~\cite{Ore}, that is, a subgroup $H$ of a group $G$ is \emph{permutable} if $HK=KH$ for every subgroup $K$ of $G$; in other words, $HK$ is a subgroup of $G$ for every subgroup $K$ of $G$.
%Clearly, normal subgroups are always permutable.
Moreover, a group $G$ is \emph{quasihamiltonian} if all its subgroups are permutable; the quasihamiltonian groups are called also \emph{Iwasawa groups} as the structure of quasihamiltonian groups was described by Iwasawa \cite{Iwa} (some gaps in the proof were filled by Napolitani~\cite{Napo}).  %For a detailed account of results concerning permutable subgroups we refer to [11].

Since for the computation of the algebraic entropy in locally finite groups we are concerned with finite subgroups, we need only the milder condition in Definition~\ref{fqhdef}. For a group $G$, let $$\mathcal{F}_C(G)=\{F\in\mathcal{F}(G)\mid FE=EF\ \text{for all}\ E\in\mathcal{F}(G)\}\subseteq \F(G)\,.$$
%A group $G$ is called \emph{quasihamiltonian} if for every pair of subgroups $X$, $Y$ of $G$ one has $XY=YX$ (equivalently, $XY$ is a subgroup of $G$ by Lemma~\ref{comm2}).

%In analogy with the quasihamiltonian groups, we introduce the following larger family of groups. For a group $G$, let
%$$\mathcal{F}_C(G)=\{F\in\mathcal{F}(G)\mid FE=EF \ \ \text{for all}\ \ E\in\mathcal{F}(G)\}\,.$$

\begin{deff}\label{fqhdef}
A group $G$ is \emph{finitely quasihamiltonian} if $\F_C(G)$ 
%$$\mathcal{F}_C(G)=\{F\in\mathcal{F}(G)\mid FE=EF\ \text{for all}\ E\in\mathcal{F}(G)\}\,.$$
is cofinal in $\mathcal{F}(G)$.
\end{deff}

Clearly, if $G$ is a quasihamiltonian group, then $\F_C(G)=\F(G)$; hence, 
\begin{center}
every quasihamiltonian group is finitely quasihamiltonian.
\end{center}
The converse is not true (e.g, $\mathbb Q/\Z\times F_2$, where $F_2$ is the non-abelian free group with two generators, is finitely quasihamiltonian but not quasihamiltonian).

%Every quasihamiltonian group is metabelian (see \cite{Iwa}), and moreover it is clear that a torsion quasihamiltonian group is locally finite. This remains true for the milder property of being finitely quasihamiltonian. %that is, every torsion finitely quasihamiltonian group is locally finite.
Clearly, every torsion finitely quasihamiltonian group is locally finite.
Another class of groups with this property is that of \emph{$FC$-groups}, that is, groups in which each element has only finitely many conjugates. Indeed, by \cite[Theorem 14.5.8]{Rob}, a group $G$ is a torsion $FC$-group if and only if $G$ is \emph{locally finite and normal}, that is, every finite subset of $G$ is contained in a normal finite subgroup of $G$:
in other words, the family of all finite normal subgroups of a group $G$, which is contained in $\F_C(G)$, is cofinal in $\Pf(G)$.
%denote by $$\mathcal N_f(G)=\{F\leq G\mid F\ \text{is normal and finite}\}$$ the family of all finite normal subgroups of a group $G$; then we have the following chain of inclusions $$\mathcal N_f(G)\subseteq\mathcal F_C(G)\subseteq \F(G)\subseteq \Pf(G).$$
%Clearly, $G$ is locally finite and normal precisely when $\mathcal N_f(G)$ is cofinal in $\Pf(G)$.
Therefore, 
\begin{center}
every torsion $FC$-group is finitely quasihamiltonian.
\end{center}

%It is worth to note that there are examples of groups, as $G=\bigoplus_{\N}\mathcal{S}_3$, that are torsion $FC$-groups, but not quasihamiltonian. 
%Analogously, for an example $H$ of a quasihamiltonian locally finite group that is not an $FC$-group see \cite{XST}. It is straightforward to verify that $G\times H$ is a finitely quasihamiltonian locally finite group which is neither an $FC$-group nor quasihamiltonian.
%Moreover, not all locally finite groups are finitely quasihamiltonian; for example, the finitary symmetric group $\mathcal{S}_f(\mathbb{N})$ over $\mathbb{N}$ is locally finite but not finitely quasihamiltonian.

%\smallskip
We summarize the relations among the classes of locally finite groups defined by the above mentioned properties in the following diagram. See Example~\ref{example} for examples witnessing that all inclusions are proper.
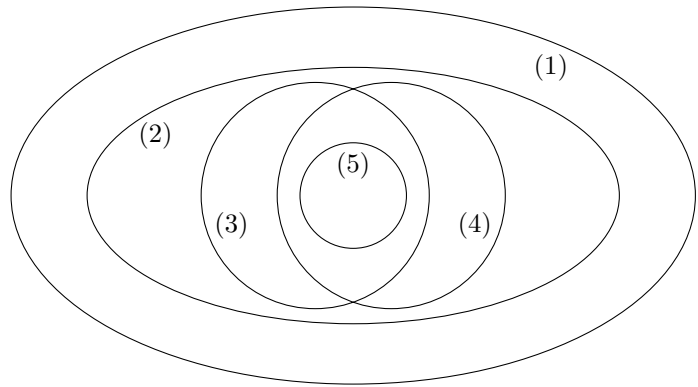
\begin{equation}\label{diagram}
\begin{tikzpicture}
\draw (-0.5,0) circle (1.5) (-2.5,-0.5)  node at (-1.6,-0.4) {(3)}
	 (0.5,0) circle (1.5) (2.5,-0.5)  node at (1.6,-0.4) {(4)}
 	 (0,0) circle (0.7) (0,-1.3)  node at (0,0.4) {(5)}
	(90:0.0cm) ellipse (3.5cm and 1.7cm) node at (-2.6,0.8) {(2)}
	(90:0.0cm) ellipse (4.5cm and 2.5cm) node [text=black] at (2.6,1.7) {(1)};
%	(-2,-2) rectangle (3,2) node [text=black,above left] {finitely quasihamiltonian }
%	(-3,-3) rectangle (4,3) node [text=black,above]{Locally finite};	
\end{tikzpicture}
\end{equation}
\begin{enumerate}[(1)]
\item Locally finite groups.
\item Finitely quasihamiltonian locally finite groups.
\item Quasihamiltonian locally finite groups.
\item Torsion $FC$-groups.
\item Torsion abelian groups.
\end{enumerate}

%In Section~\ref{fqhsec} we further discuss the relations among these properties, giving also some examples. Moreover, we show some properties of finitely quasihamiltonian locally finite groups that witness why in this class one can compute the algebraic entropy more easily that in the whole class of locally finite groups.
%\begin{rem}\label{cofin}
Going back to the algebraic entropy, we see that in the class of finitely quasihamiltonian locally finite groups we can compute the algebraic entropy more easily than in the whole class of locally finite groups. Indeed, for a quasihamiltonian locally finite group $G$, we have that $\mathcal{F}_C(G)$ is cofinal in $\Pf(G)$, %Therefore, $\mathcal{F}_C(G)$ is cofinal in $
so, for $\phi\in\End(G)$, we can compute (see Remark~\ref{rem1})
\begin{equation*}
h(\phi)=\sup\{H(\phi,F)\mid F\in\mathcal{F}_C(G)\}\,,
\end{equation*}
%\end{rem}
where, for every $F\in\F_C(G)$, each $T_n(\phi,F)$ is a subgroup of $G$ (see Lemma~\ref{comm}).

\medskip
Our main result is the following Addition Theorem for locally finite groups that are finitely quasihamiltonian.

\begin{teo}\label{atq}
If $G$ is a finitely quasihamiltonian locally finite group, then $AT(G)$ holds.
%Let $G$ be a finitely quasihamiltonian locally finite group, $\phi\in\End(G)$, $H$ a
%$\phi$-invariant normal subgroup of G and $\bar \phi\in \End(G/H)$ the endomorphism induced by $\phi$. Then
%$$ h(\phi)=h(\phi\rest_H)+h(\bar \phi)\,.$$
\end{teo}

Its proof is divided in the proof of two inequalities. The easier one (see Proposition~\ref{first}) is given in Section~\ref{fqhsec}, where we also show useful properties of finitely quasihamiltonian locally finite groups for the computation of the algebraic entropy. The proof of the second inequality (see Proposition~\ref{second}) is based first of all on the passage from the sequence $({\ell(T_{n}(\phi,X))}/{n})_{n\in\N}$, used in \eqref{limite} for the definition of algebraic entropy, to its subsequence $({\ell(T_{2^n}(\phi,X))}/{2^n})_{n\in\N}$, which turns out to be decreasing (see Proposition~\ref{6.8a}). Moreover, in Section~\ref{ellsec}, we use the auxiliary function $\ell(-,-)$ introduced in \cite{DFG} in a much more general context.

\smallskip
The following is a direct consequence of Theorem~\ref{atq}.

\begin{coro}\label{ATcor}
If $G$ is a locally finite group which is either an $FC$-group or quasihamiltonian, then $AT(G)$ holds.
%In particular, if $G$ is a torsion abelian group, then $AT(G)$ holds.
\end{coro}

Clearly, this result extends a consequence of \cite[Corollary 7.2]{GBST} stating that $AT(G,\phi,H)$ holds for every torsion $FC$-group $G$, every $\phi\in\End(G)$ and every $\phi$-invariant normal subgroup $H$ of $G$ with $\phi\restriction_H$ surjective and $\bar\phi_{G/H}$ injective. 
Moreover, it extends one of the main results from \cite{XST}, namely, that $AT(G)$ holds for every locally finite group $G$ which is a quasihamiltonian $FC$-group.
Indeed, Theorem~\ref{atq} covers the family ($2$) in the above diagram, and Corollary~\ref{ATcor}  the union of ($3$) and ($4$), while the result from \cite{GBST} concerns the groups in ($4$) and the result from \cite{XST} the groups in the intersection of ($3$) and ($4$).

\smallskip
The validity of the Addition Theorem for the class ($5$) corresponds to the above mentioned result from \cite{DGSZ}:

\begin{coro}
If $G$ is a torsion abelian group, then $AT(G)$ holds.
\end{coro}

We underline that the proof of Theorem~\ref{atq} presented in this paper was inspired by ideas contained in \cite{DFG}. Moreover, it is much shorter and simpler than that in \cite{DGSZ}, and it follows a different path. The proofs of the mentioned results from \cite{GBST} and \cite{XST} use a third different approach inspired by that in  \cite{CGBalg,CGBtop,GBVtop,SV}, based on the so-called Limit-free Formula (see \cite{CGBalg,DGBlf,GBtdlc,Y}).

%\begin{coro}
%%Let $G$ be a locally finite and normal group, $\phi\in\End(G)$, $H$ a
%%$\phi$-invariant normal subgroup of G and $\bar \phi\in\End(G/H)$ the endomorphism induced by $\phi$. Then
%%$$ h(\phi)=h(\phi\rest_H)+h(\bar \phi)\,.$$
%\end{coro}
%\begin{proof}
%Follows by Lemma~\ref{lfn} and Theorem~\ref{atq}.
%\end{proof}

%\begin{coro}
%Let $G$ be a locally finite quasihamiltonian group, $\phi\in\End(G)$, $H$ a
%$\phi$-invariant normal subgroup of G and $\bar \phi\in \End(G/H)$ the endomorphism induced by $\phi$. Then
%$$ h(\phi)=h(\phi\rest_H)+h(\bar \phi)\,.$$
%\end{coro}
%\begin{proof}
%Follows by Remark~\ref{qht} and Theorem~\ref{atq}.
%\end{proof}

\subsection*{Acknowledgements}

It is a pleasure to thank Dikran Dikranjan and Pablo Spiga for their useful comments and suggestions.

\smallskip
The first named author was partially supported by the project ``Topological, Categorical and Dynamical Methods in Algebra'' (ToCaDyMA) funded by the Department of Mathematics, Computer Science and Physics of the University of Udine, and was partially supported also by the National Group for Algebraic and Geometric Structures, and their Applications (GNSAGA – INdAM).

\section{Finitely quasihamiltonian locally finite groups}\label{fqhsec}

First of all we give examples showing that the inclusions among the classes in the diagram \eqref{diagram} above are all proper. In particular we see that not all locally finite groups are finitely quasihamiltonian.

\begin{esem}\label{example}
\begin{enumerate}[(a)]
\item The quaternion group $Q_8$ is non abelian, but it is a quasihamiltonian torsion $FC$-group. This means that ($5$) is properly contained in the intersection of ($3$) and ($4$).

\item All finite groups are torsion $FC$-groups, that is, the class of finite groups is contained in ($4$).
In particular, the symmetric group $S_3$ is in ($4$), but it is not in $(3)$ as it is not quasihamiltonian.

\item In \cite{XST} the following example was given of a quasihamiltonian locally finite group $H$ that is not an $FC$-group; this means that $H$ is in ($3$), but it is not in $(4)$.

Let $H=\Z_{3^2}^\N\rtimes_\alpha\Z_3$, where $\alpha$ is the action of $\Z_3$ on $\Z_{9}^\N$ defined, for every $x\in\Z_3$ and every $a\in\Z_9^\N$, by $\alpha(x)(a)=4^xa$. Since $H$ is a non-abelian $3$-group that satisfies \cite[Theorem 3]{Iwa}, so $H$ is quasihamiltonian; but $H'$ is infinite and so $H$ cannot be an $FC$-group by \cite[Proposition 2.8]{XST}.

\item As a consequence of the previous items, the group $S_3\times H$ is a finitely quasihamiltonian locally finite group which is neither an $FC$-group nor quasihamiltonian. So, the union of ($3$) and ($4$) is properly contained in ($2$).

\item The finitary symmetric group $\mathcal S_{fin}(\N_+)$ is locally finite but not finitely quasihamiltonian. So, ($2$) is properly contained in ($1$).

In fact, $\mathcal S_{fin}(\N_+)=\bigcup_{n\in\N_+}\mathcal S_n$, where $\mathcal S_n=\{\sigma\in\mathcal S_{fin}(\N)\mid \mathrm{supp}(\sigma)=\{1,\ldots,n\}\}$, so $\mathcal S_{fin}(\N_+)$ is locally finite. Now assume that, for a fixed $n\in\N$ with $n>1$, there exists $H\in\mathcal{F}_C(\mathcal{S}_{fin}(\mathbb{N_+}))$ that contains $\mathcal{S}_n$. Since $H$ is finite, there exists $m\in\N_+$ such that $H\subseteq\mathcal{S}_m$. Consider the cyclic subgroup $N=\langle\tau\rangle$, where $\tau=( n\ m+1)$. Since $H\in\mathcal{F}_C(\mathcal{S}_{fin}(\mathbb{N}_+))$, we have that $HN=NH$, and so there exists $\sigma\in H$ such that $(1\ n\ m+1)=(1\ n)\tau=\tau\sigma$. We conclude that necessarily $\sigma(1)=m+1$, but this is absurd because $\sigma\in H\subseteq \mathcal{S}_m$.
\end{enumerate}
\end{esem}

The next proposition in particular shows that the class of finitely quasihamiltonian locally finite groups is stable under taking subgroups and quotients. The technical conditions will be useful in the computation of the algebraic entropy with respect to Remark~\ref{rem1}.

\begin{pro}\label{norloc}
Let $G$ be a finitely quasihamiltonian group and $H$ a subgroup of $G$. Then:
\begin{enumerate}[(a)]
\item $\bar{\mathcal{F}}_C(H)=\{F\cap H\mid F\in\mathcal{F}_C(G)\}\subseteq \mathcal{F}_C(H)$ is cofinal in $\F(H)$, and in particular $H$ is finitely quasihamiltonian; 
\item if $G$ is locally finite and $H$ is normal in $G$, then $\bar{\mathcal{F}}_C(G/H)=\{\pi(F)\mid F\in\mathcal{F}_C(G)\}\subseteq\mathcal{F}_C(G/H)$ is cofinal in $\F(G/H)$, and in particular $G/H$ is finitely quasihamiltonian. 
\end{enumerate}
\end{pro}
\begin{proof}
Since $G$ is finitely quasihamiltonian $\mathcal{F}_C(G)$ is cofinal in $\mathcal{F}(G)$.

(a) Since $\mathcal{F}_C(G)$ is cofinal in $\mathcal{F}(G)$, $\bar{\mathcal{F}}_C(H)$ is cofinal in $\mathcal{F}(H)$. So, it remains to show that $$\bar{\mathcal{F}}_C(H)\subseteq\mathcal{F}_C(H)\,.$$ Fix $F\in\mathcal{F}_C(G)$ and $E\in\mathcal{F}(H)$. Then $E\in\mathcal{F}(G)$ and therefore $EF=EF$. %So for every $f\in F\cap H$ and $e\in E$, there are $f'\in F$ and $e'\in E$ such that $fe=e'f'$. Since $e,\ f,\ e'\in H$ we have that $e'^{-1}fe=f'\in F\cap H$, and so
%$$(F\cap H)E\subseteq E(F\cap H)\,.$$
%Analogously we can prove the other inclusion.
Since $E\subseteq H$, it is straightforward to verify that also $(F\cap H)E=E(F\cap H)$. Hence, $F\cap H\in\F_C(H)$.

(b) First note that $\bar{\mathcal{F}}_C(G/H)$ is cofinal in $\mathcal{F}(G/H)$. Indeed, let $\bar F\in \mathcal{F}(G/H)$. Since $G$ is locally finite, there exists $F\in\F(G)$ such that $\pi(F)=\bar F$. Since $G$ is finitely quasihamiltonian, there exists $\tilde F\in \mathcal{F}_C(G)$ such that $F\subseteq \tilde F$, and so $\bar F\subseteq \pi(\tilde F)$.

It remains to show that $$\bar{\mathcal{F}}_C(G/H)\subseteq\mathcal{F}_C(G/H)\,.$$  Fix $F\in\mathcal{F}_C(G)$ and $\bar E\in\mathcal{F}(G/H)$. 
%Let $f\in F$ and $e\in E$. Consider now $\bar e\in G$ such that $\pi(\bar e)=e$, and $\bar E=\langle\bar e\rangle$.  By hypothesis $G$ is torsion, so $F\bar E=\bar EF$ and there exist $f'\in F$ and $k\in\mathbb {N}$ such that $\bar e f=f'\bar {e}^k$.\\
%Now we have
%$$ e\pi(f)=\pi(\bar ef)=\pi(f'\bar {e}^k)=\pi(f')e^k\,,$$
%and so $E\pi(F)\subseteq\pi(F)E$.
%Analogously we can prove the other inclusion.
%By Lemma~\ref{lfs} 
Since $G$ is locally finite, there exists $E\in\F(G)$ such that $\pi(E)=\bar E$. Therefore, $FE=EF$, and so $\pi(F)\bar E=\bar E\pi(F)$.
Hence, $\pi(F)\in\F_C(G/H)$.
\end{proof}

We recall a useful basic property of the algebraic entropy that allows us to compute the algebraic entropy taking in account cofinal subfamilies of $\Pf(G)$ (e.g., $\F_C(G)$ when $G$ is finitely quasihamiltonian and locally finite).

\begin{rem}\label{rem1}
Let $G$ be a group and $\phi\in\End(G)$.
The function $H(\phi,-):\Pf(G)\to\R_{\geq0}$ is non-decreasing, that is, if $X,X'\in \Pf(G)$ and $X\subseteq X'$, then $H(\phi,X)\leq H(\phi,X')$. Therefore, if $\mathcal{F}$ is a cofinal subfamily of $\Pf(G)$ with respect to the inclusion, then
$$h(\phi)=\text{sup}\{H(\phi,X)\mid X\in\mathcal{F}\}\,.$$
\end{rem}

Next we see that, fixed $\phi\in\End(G)$, for every $F\in\mathcal{F}_C(G)$, each $T_n(\phi,F)$ is a subgroup of $G$.

\begin{lemma}\label{comm}
Let $G$ be a group, $\phi\in\End(G)$ and $F\in \mathcal{F}_C(G)$. Then for all $n,m\in\mathbb{N}$
$$\phi^{n}(F)\phi^m(F)=\phi^m(F)\phi^{n}(F)\,.$$
Consequently, $T_n(\phi, F)$ is a subgroup of $G$ for all $n\in\mathbb{N}$.
\end{lemma}
\begin{proof}
Since $F\in \mathcal{F}_C(G)$ we have that $F\phi^k(F)=\phi^k(F)F$ for all $k\in\mathbb{N}$. Therefore, fixed $m$ and $n\in\mathbb{N }$ with $n\leq m$,
\begin{equation*}
\phi^{n}(F)\phi^m(F)=\phi^{n}(F\phi^{m-n}(F))=\phi^n(\phi^{m-n}(F)F)=\phi^m(F)\phi^{n}(F)\,.
\end{equation*}
%\end{proof}
%\begin{coro}\label{subg}
%Let $G$ be a finitely quasihamiltonian group, $\phi\in\End(G)$ and consider $F\in \mathcal{F}_C(G)$. Then $T_n(\phi, F)\in\mathcal F(G)$ for all $n\in\mathbb{N}$.
%\end{coro}
%\begin{proof}

To prove the second assertion, we proceed by induction. For $n\in\{0,1\}$ the assertion is verified. %Moreover, since $F\in \mathcal{F}_C(G)$,
%$$T_2(\phi,F)=F\phi(F)=\phi(F)F,$$
%and so $T_2(\phi,F)$ is a subgroup of $G$.
Suppose that $T_n(\phi, F)$ is a subgroup of $G$ for some $n\in\N_+$; then, by the first part of the lemma,
$$T_{n+1}(\phi,F)=T_n(\phi,F)\phi^n(F)=\phi(F)^nT_n(\phi,F),$$
and so $T_{n+1}(\phi,F)$ is a subgroup of $G$.
\end{proof}

Thanks to Proposition~\ref{norloc}, Remark~\ref{rem1} and Lemma~\ref{comm}, we can immediately prove one of the inequalities needed in Theorem~\ref{atq}.
%Indeed, if $G$ is a finitely quasihamiltonian group, by definition $\mathcal{F}_C(G)$ is cofinal in $\mathcal{F}(G)$. %and so
%%\begin{equation}
%%\ent(\phi)=\sup\{H(\phi,F):F\in\mathcal{F}_C(G)\}\,,
%%\end{equation}
%If $G$ is also locally finite, we have that $\mathcal{F}(G)$ is cofinal in $\Pf(G)$. Therefore, $\mathcal{F}_C(G)$ is cofinal in $\Pf(G)$ and then
%\begin{equation}
%h(\phi)=\sup\{H(\phi,F):F\in\mathcal{F}_C(G)\}\,.
%\end{equation}

\begin{pro}\label{first}
Let $G$ be a finitely quasihamiltonian locally finite group, $\phi\in\End(G)$, and $H$ a $\phi$-invariant normal subgroup of $G$. Then
$$h(\phi)\geq h(\phi\rest_H)+h(\bar \phi_{G/H})\,.$$
\end{pro}
\begin{proof}
Let $\pi:G\rightarrow G/H$ be the canonical projection. %By Proposition~\ref{norloc} $H$ and $G/H$ are finitely quasihamiltonian locally finite groups. 
Fix $F,E\in\F_C(G)$ and consider $F\cap H\in \bar\F_C(H)$ and $\pi(E)\in \mathcal{F}_C(G/H)$.
Let $B=FE$, $A=B\cap H$, and $C=\pi(B)$; then $B\in\F_C(G)$, $F\cap H\subseteq A\in \bar\F_C(H)$, and $\pi(E)\subseteq C\in\bar\F_C(G/H)$.

Since $\pi(T_n(\phi,B))=T_n(\bar\phi_{G/H},\pi(B))=T_n(\bar\phi_{G/H},C)$, the exact sequence $$0\rightarrow A\rightarrow B\rightarrow C\rightarrow 0$$ gives rise, for every $n\in\N$, to the sequence
$$T_n(\phi\rest_H,A)\rightarrow T_n(\phi,B)\rightarrow T_n(\bar\phi_{G/H}, C)\,,$$
where all are subgroups in view of Lemma~\ref{comm}.
%This sequence hasn't to be exact anymore, but we have that
As $T_n(\phi\rest_H,A)\subseteq \ker(\pi\rest_{T_n(\phi,B)})$,
we have that
$$\ell(T_n(\phi\rest_H,A)+\ell(T_n(\bar\phi_{G/H},C))\leq \ell(\ker(\pi\rest_{T_n(\phi,B)}))+ \ell(T_n(\bar\phi_{G/H},C))=\ell(T_n(\phi,B))\,.$$
Dividing by $n$ and taking the limit, we conclude that
$$H(\phi\rest_H, F\cap H)+ H(\bar\phi_{G/H}, \pi(E))\leq H(\phi\rest_H, A)+H(\bar\phi_{G/H},C)\leq H(\phi,B)\leq h(\phi)\,.$$
%Taking the supremum over all $Z\in \mathcal{F}_C(G/H)$ and $X\in \mathcal{F}_C(H)$, by Remark~\ref{cofin} we get the thesis.
Since $E,F\in\F_C(G)$ where chosen arbitrarily, and in view of Proposition~\ref{norloc} and Remark~\ref{rem1}, we get the thesis.
\end{proof}

\section{Proof of the Addition Theorem}\label{ellsec}
%{The function $\ell(-,-)$}\label{ellsec}

We start showing that in order to compute the algebraic entropy $H(\phi,X)$ we can choose the suitable subsequence $({\ell(T_{2^n}(\phi,X))}/{2^n})_{n\in\N}$ of $({\ell(T_{n}(\phi,X))}/{n})_{n\in\N}$, which is decreasing.

\begin{pro}\label{6.8a}
Let $G$ be a group, $\phi\in\End(G)$ and $X\in\Pf(G)$ with $1\in X$. Then:
\begin{enumerate}[(a)]
\item the function
\begin{equation*}
n\mapsto \frac{\ell(T_{2^n}(\phi,X))}{2^n}
\end{equation*}
is decreasing;
\item $H(\phi,X)=\inf_{n\in\mathbb{N}} \frac{\ell(T_{2^n}(\phi,X))}{2^n}$.
\end{enumerate}
\end{pro}
\begin{proof}
(a) Let $n\in\mathbb{N}$. Since $T_{2^{n+1}}(\phi,X))=T_{2^n}(\phi,X)\phi^{2^n}(T_{2^n}(\phi,X))$, we have that
\begin{equation*}
\ell(T_{2^{n+1}}(\phi,X))=\ell(T_{2^n}(\phi,X)\phi^{2^n}(T_{2^n}(\phi,X)))\leq \ell(T_{2^n}(\phi,X))+\ell(\phi^{2^n}T_{2^n}(\phi,X))
\leq 2\ell(T_{2^n}(\phi,X))\,.
\end{equation*}
Dividing both sides by $2^{n+1}$, we obtain
$$\frac{\ell(T_{2^{n+1}}(\phi,X))}{2^{n+1}}\leq\frac{\ell(T_{2^n}(\phi,X))}{2^n}\,.$$

(b) Since $({\ell(T_{2^n}(\phi,X))}/{2^n})_{n\in\N}$ is a subsequence of $({\ell(T_{n}(\phi,X))}/{n})_{n\in\N}$, by \eqref{limite} we have that
\begin{equation*}
H(\phi,X)=\lim_{n\to\infty} \frac{\ell(T_{2^n}(\phi,X))}{2^n}=\inf_{n\in\mathbb{N}} \frac{\ell(T_{2^n}(\phi,X))}{2^n}\,.
\end{equation*}
where the last equality holds by item (a).
\end{proof}

Given a group $G$, denote by $\mathcal L(G)$ the lattice of all subgroups of $G$. Let $$\ell: \mathcal{P}(G)\times \mathcal{L}(G)\rightarrow \mathbb{R}_{\geq0}\cup\{\infty\},\quad \ell(X,B)=\log[XB:B]=\log|\{xB\mid x\in B\}|\,;$$ 
%Observe that $\ell(X,\{1\})=\ell(X)$.
in other words, letting $\pi:G\rightarrow \{xB\mid x\in G\}$ be 
%\begin{rem}\label{pi}
%Let $G$ be a group, $B\in\mathcal{L}(G)$ and $X\in\mathcal{P}(G)$. Consider 
the canonical projection, we have that 
\begin{equation}\label{pi}
\ell(X,B)=\ell(\pi(X))\,.
\end{equation}
%\end{rem}

We collect in the following lemma the useful properties of the function $\ell(-,-)$.
%study the useful properties of the function $\ell(-,-)$.

%\begin{lemma}\label{sep}
%Let $G$ be a group, $B\in\mathcal{L}(G)$ and $X,X'\in\mathcal{P}(G)$. Then $$\ell(XX',B)\leq \ell(X,B)+\ell(X',B)\,.$$
%\end{lemma}
%\begin{proof}
%It suffices to compute that
%$$\ell(XX',B)=\ell(\pi(XX'))\leq\ell(\pi(X))+\ell(\pi(X'))=\ell(X,B)+\ell(X',B)\,,$$
%where the first and the last equality follow by Remark~\ref{pi}.
%\end{proof}

\begin{lemma}\label{6.3}
Let $G$ be a group, $X,X'\in \mathcal{P}(G)$ and $B, B'\in\mathcal{L}(G)$. Then:
\begin{enumerate}[(a)]
\item\label{6.3.1} the function $\ell(X,B)$ is increasing in $X$ and decreasing in $B$;
\item \label{6.3.2}$\ell(XB)=\ell(X,B)+\ell(B)$;
\item\label{sep} $\ell(XX',B)\leq \ell(X,B)+\ell(X',B)$;
%\item \label{6.3.3}if $X'B$ is a subgroup, $\ell(XX',B)=\ell(X,X'B)+\ell(X',B)$;
\item \label{6.3.4}if $BB'$ is a subgroup, $\ell(XX',BB')\leq \ell(X,B)+\ell(X',B')$;
\item \label{enn} for $\phi\in\End(G)$, $\ell(\phi(X),\phi(B))\leq\ell(X,B)$.
\end{enumerate}
\end{lemma}
\begin{proof}
\eqref{6.3.1} If $X\subseteq X'$, then $XB\subseteq X'B$ and so $\ell(X,B)\leq\ell(X',B)$. If $B'\subseteq B$, then $|\{xB\mid x\in X\}|\leq |\{xB'\mid x\in X\}|$.

\eqref{6.3.2} Since $|XB|=[XB:B]|B|$, we have that
\begin{equation*}
\ell(XB)=\log[XB:B]|B|\leq\log[XB:B]+\log|B|=\ell(X,B)+\ell(B)\,.
\end{equation*} 

\eqref{sep} It suffices to compute that, for $\pi:G\to \{xB\mid x\in G\}$ the canonical projection,
$$\ell(XX',B)=\ell(\pi(XX'))\leq\ell(\pi(X))+\ell(\pi(X'))=\ell(X,B)+\ell(X',B)\,,$$
where the first and the last equality follow from \eqref{pi}.
%(c) Since $$\ell(XX',B)=\ell(XX'B,B)\quad\text{and}\quad[XX'B:B]=[XX'B:X'B][X'B:B]\,,$$ we conclude that
%\begin{equation*}
%\begin{split}
%\ell(XX',B)=\log[XX'B:B]=\log[XX'B:X'B]+\log[X'B:B]=\ell(X,X'B)+\ell(X',B)\,.
%\end{split}
%\end{equation*}

\eqref{6.3.4} Follows from \eqref{sep} and \eqref{6.3.1}. 
%we conclude that
%\begin{equation*}
%\ell(XX', BB')\leq\ell(X,BB')+\ell(X',BB')\leq\ell(X,B)+\ell(X',B')\,,
%\end{equation*}
%and we are done.

\eqref{enn} The map $\{xB\mid x\in X\}\to \{\phi(x)\phi(B)\mid x\in X\}$ induced by $\phi$ is well-defined and surjective.
\end{proof}

%\begin{lemma}\label{inter}
%Let $G$ be a group and $A,B\in \mathcal{L}(G)$. Then $\ell(A,B)=\ell(A,A\cap B)$.
%\end{lemma}
%\begin{proof}
%Since $[AB:B]=[A:A\cap B]$, we have that
%\begin{equation*}
%\ell(A,B)=\log[AB:B]=\log[A:A\cap B]=\ell(A,A\cap B)\,.
%\end{equation*}
%hence the thesis.
%\end{proof}

%\begin{lemma}\label{enn}
%Let $G$ be a group, $\phi\in\End(G)$, $X\in\mathcal{P}(G)$ and $B\in\mathcal{L}(G)$. Then $\ell(\phi(X),\phi(B))\leq\ell(X,B)$. %for every $n\in\mathbb{N}$.
%\end{lemma}
%\begin{proof}
%It suffices to note that $\phi$ induces a surjective map $\{xB\mid x\in X\}\to \{\phi(x)\phi(B)\mid x\in X\}$.
%\end{proof}

%\section{Proof of the Addition Theorem}\label{ATsec}

The next proposition shows how the function $\ell(-,-)$ allows us to compute the algebraic entropy of $\bar\phi_{G/H}$ remaining in some sense inside $G$.

%\begin{pro}\label{6.2}
%Let $G$ be a group, $\phi\in\End(G)$, $H$ a $\phi$-invariant normal subgroup of $G$ and $\pi :G\rightarrow G/H$ the canonical projection. Let $n\in\mathbb{N}$ and $Y\in \Pf(G)$ with $1\in Y$. Then:
%\begin{enumerate}[(a)]
%\item $\ell(T_n(\bar \phi, \pi(Y)))=\ell(T_n(\phi,Y),H)$;
%\item $H(\bar\phi_{G/H},\pi(Y))=\lim_{n\to\infty}\frac{\ell(T_n(\phi,Y),H)}{n}$.
%\end{enumerate}
%\end{pro}
%\begin{proof}
%(a) Since $\pi(T_n(\phi,Y))=T_n(\bar \phi,\pi(Y))$, it suffices to apply \eqref{pi}.
%%\begin{equation}\label{eq4.2}
%%%\begin{split}
%%\pi(T_n(\phi,Y))=\pi(\prod_{i=0}^{n-1}\phi^i(Y))=\prod_{i=0}^{n-1}\pi\circ \phi^i(Y)=\prod_{i=0}^{n-1}\bar{\phi}^i\circ\pi(Y)=T_n(\bar \phi,\pi(Y))\,.
%%%\end{split}
%%\end{equation}
%%we have that
%%$$\pi(T_n(\phi,Y))\cong\frac{T_n(\phi,Y)H}{H}\,,$$
%%and this concludes the proof.
%
%(b) Follows from item (a) and the definition.
%\end{proof}

%The following is a direct consequence of Proposition~\ref{6.8a} and Proposition~\ref{6.2}.

\begin{pro}\label{6.2cor}
Let $G$ be a group, $\phi\in\End(G)$, $H$ a $\phi$-invariant normal subgroup of $G$ and $\pi :G\rightarrow G/H$ the canonical projection. Let $n\in\mathbb{N}$ and $X\in \Pf(G)$ with $1\in X$. Then:
\begin{enumerate}[(a)]
\item the function
\begin{equation*}
n\mapsto \frac{\ell(T_{2^n}(\phi,X),H)}{2^n}
\end{equation*}
is decreasing;
\item $H(\bar\phi_{G/H},\pi(X))=\inf_{n\in\N}\frac{\ell(T_{2^n}(\phi,X),H)}{2^n}$%$H(\phi,X)=\inf_{n\in\mathbb{N}} \frac{\ell(T_{2^n}(\phi,X))}{2^n}$.
\end{enumerate}
\end{pro}
\begin{proof}
Since, for every $n\in\N$, $T_n(\bar \phi_{G/H},\pi(X))=\pi(T_n(\phi,X))$, by \eqref{pi} we have that 
$$\ell(T_n(\bar \phi_{G/H}, \pi(X)))=\ell(T_n(\phi,X),H)\,.$$
Now the statements in (a) and (b) follow from the latter equation and respectively from items (a) and (b) of  Proposition~\ref{6.8a} applied to $\bar\phi_{G/H}$ and $\pi(X)$.
%(b) By definition, $$H(\bar\phi_{G/H},\pi(X))=\lim_{n\to\infty}\frac{\ell(T_n(\phi,X),H)}{n}.$$
\end{proof}

%\begin{rem}
%If $X\in\mathcal{P}^1_{fin}$, then if $n\leq m$ we have that $T_n(\phi,X)\subseteq T_m(\phi,X)$. By Lemma~\ref{6.3}(\ref{6.3.1}) we conclude that
%$$\ell(T_n(\phi,X),H)\leq\ell(T_m(\phi,X),H)\,.$$
%If $G$ is finitely quasihamiltonian and $H\in\mathcal{F}_C(G)$, by Lemma~\ref{subg} we have that $T_n(\phi,H)$ is a finite subgroup of $G$ for all $n\in\mathbb{N}$. Again, by  Lemma~\ref{6.3}(\ref{6.3.1}) we find that
%$$\ell(T_n(X,\phi,H))\leq\ell(X,T_m(\phi,H))\,.$$
%\end{rem}

%\begin{lemma}\label{6.4}
%Let $G$ be a finitely quasihamiltonian group, $\phi\in\End(G)$, $X\in\mathcal{P}^1_{fin}(G)$ and $H\in\mathcal{F}_C(G)$. Then:
%\begin{enumerate}[(a)]
%\item $\ell(T_n(\phi,X),T_n(\phi,H))\leq n\ell(X,H)$;
%\item if $H$ is $\phi$-invariant, then $\ell(T_n(\phi,X),H)\leq n\ell(X,H)$.
%\end{enumerate}
%\end{lemma}
%\begin{proof}
%By items \eqref{6.3.1} and \eqref{6.3.4} of Lemma~\ref{6.3} we have that
%\begin{equation*}
%\begin{split}
%\ell(T_n(\phi,X),T_n(\phi,H))&\leq \sum_{i=0}^{n-1}\ell(\phi^i(X),T_n(\phi,H))\leq\\
%&\leq\sum_{i=0}^{n-1}\ell(\phi^i(X),\phi^i(H))\leq \sum_{i=0}^{n-1}\ell(X,H)=n\ell(X,H)\,.
%\end{split}
%\end{equation*}
%If $H$ is also $\phi$-invariant, we have that $T_n(\phi,H)\leq H$. By item (a) and Lemma~\ref{6.3}\eqref{6.3.1} we obtain that
%$$\ell(T_n(\phi,X),H)\leq\ell(T_n(\phi,X),T_n(\phi,H))\leq n\ell(X,H)\,.$$
%\end{proof}

The following technical lemma is needed in the last and main proof of the paper.
We choose $F\in\F_C(G)$ in order to have that each $T_n(\phi,F)$ is a subgroup of $G$ in view of Lemma~\ref{comm}.

\begin{lemma}\label{6.8}
Let $G$ be a group, $\phi\in\End(G)$, $X\in\Pf(G)$ with $1\in X$, and $F\in \mathcal{F}_C(G)$. Then the function
\begin{equation*}
n\mapsto \frac{\ell(T_{2^n}(\phi,X),T_{2^n}(\phi, F))}{2^n}
\end{equation*}
is decreasing.
\end{lemma}
\begin{proof}
Let $n\in\mathbb{N}$. By Lemma~\ref{6.3}(d,e), we have that
\begin{equation*}
\begin{split}
\ell(T_{2^{n+1}}(\phi,X),T_{2^{n+1}}(\phi,F))&= \ell(T_{2^n}(\phi,X)\phi^{2^n}(T_{2^n}(\phi,X)),T_{2^n}(\phi,F)\phi^{2^n}(T_{2^n}(\phi,F)))\\
&\leq \ell(T_{2^n}(\phi,X),T_{2^n}(\phi,F))+\ell(\phi^{2^n}T_{2^n}(\phi,X),\phi^{2^n}T_{2^n}(\phi,F))\\
&\leq 2\ell(T_{2^n}(\phi,X),T_{2^n}(\phi,F))\,.
\end{split}
\end{equation*}
Dividing both sides by $2^{n+1}$, we obtain
$$\frac{\ell(T_{2^{n+1}}(\phi,X),T_{2^{n+1}}(\phi,F))}{2^{n+1}}\leq\frac{\ell(T_{2^n}(\phi,X),T_{2^n}(\phi,F))}{2^n}\,,$$
that concludes the proof.
\end{proof}

Now we are in position to complete the proof of Theorem~\ref{atq} by showing that also the converse inequality with respect to that in Proposition~\ref{first} holds true.

%\section{Addition Theorem}

%\begin{teo}[Addition Theorem]\label{atq}
%Let $G$ be a locally finite finitely quasihamiltonian group, $\phi\in \End(G)$, $H$ a
%$\phi$-invariant normal subgroup of G and $\bar \phi\in \End(G/H)$ the endomorphism induced by $\phi$. Then
%$$ h(\phi)=h(\phi\rest_H)+h(\bar \phi)\,.$$
%\end{teo}
\begin{pro}\label{second}
Let $G$ be a finitely quasihamiltonian locally finite group, $\phi\in\End(G)$, and $H$ a $\phi$-invariant normal subgroup of $G$. Then
$$h(\phi)\leq h(\phi\rest_H)+h(\bar\phi_{G/H})\,.$$
\end{pro}
\begin{proof}
Consider $D\in \mathcal{F}_C(G)$ and let $C=\pi(D)\in \bar\F_C(G/H)$. Fix $\varepsilon>0$. By Proposition~\ref{6.2cor}, there exists $M\in\mathbb{N}$ such that, for every $n\geq M$,
%\begin{equation}\label{equ1}
%\frac{\ell(T_{2^n}(\phi,D))}{2^n}\leq H(\phi,D)+\varepsilon
%\end{equation}
%and
\begin{equation}\label{equ2}
\frac{\ell(T_{2^n}(\phi,D),H)}{2^n}\leq H(\bar\phi_{G/H},C)+\varepsilon\,.
\end{equation}
Let $$T=T_{2^M}(\phi,D)\quad \text{and}\quad H_0=H\cap T\,.$$ 
Since $H$ is finitely quasihamiltonian by Proposition~\ref{norloc}, there exists $\bar H_0\in \mathcal{F}_C(H)$ that contains $H_0$; let $$S=T_{2^M}(\phi,\bar H_0)\,.$$
By definition, $H_0=H\cap T\subseteq T$, and so $$\ell(T,H_0)=\log[T:H_0]=\log[T:H\cap T]=\log[TH:H]=\ell(T,H)\,.$$ %Since $[AB:B]=[A:A\cap B]$, we have that
%\begin{equation*}
%\ell(A,B)=\log[AB:B]=\log[A:A\cap B]=\ell(A,A\cap B)\,.
Moreover, $\ell(T,H)\leq\ell(T,S)\leq\ell(T,\bar H_0)\leq\ell(T,H_0)$ by Lemma~\ref{6.3}\eqref{6.3.1}. Hence, we have that
\begin{equation}\label{equ3}
\ell(T,S)=\ell(T,H)\,.
\end{equation}
Let $n\geq M$. By Lemma~\ref{6.8} and equations \eqref{equ2} and \eqref{equ3} we have that
\begin{equation}\label{equ4}
\begin{split}
\frac{\ell(T_{2^n}(\phi,D),T_{2^n}(\phi,\bar H_0))}{2^n}&\leq\frac{\ell(T_{2^M}(\phi,D),T_{2^M}(\phi,\bar H_0))}{2^M}=\frac{\ell(T,S)}{2^M}=\\
&=\frac{\ell(T,H)}{2^M}\leq H(\bar\phi_{G/H}, C)+\varepsilon\leq h(\bar\phi_{G/H})+\varepsilon\,.
\end{split}
\end{equation}
By Proposition~\ref{6.8a}(a,b) there exists $M'\geq M$, such that for every $n\geq M'$,
\begin{equation}\label{equ5}
\frac{\ell(T_{2^n}(\phi,\bar H_0))}{2^n}\leq H(\phi\rest_H,\bar H_0)+\varepsilon\,\leq h(\phi\rest_H)+\varepsilon\,.
\end{equation}
Let $n\geq M'$. By Proposition~\ref{6.8a}(b), 
\begin{equation}\label{equ7}
H(\phi,D)\leq\frac{\ell(T_{2^n}(\phi,D))}{2^n}\,.
\end{equation}
Moreover, by Lemma~\ref{6.3}(\ref{6.3.1},\ref{6.3.2}), we have that
\begin{equation}\label{equ6}
\ell(T_{2^n}(\phi,D))\leq \ell(T_{2^n}(\phi,D)T_{2^n}(\phi,\bar H_0))=\ell(T_{2^n}(\phi,D),T_{2^n}(\phi,\bar H_0))+\ell(T_{2^n}(\phi,\bar H_0)).
\end{equation}
Hence, by \eqref{equ7}, \eqref{equ6}, \eqref{equ4}, and \eqref{equ5}, we obtain that
\begin{equation*}
%\begin{split}
H(\phi,D)\leq\frac{\ell(T_{2^n}(\phi,D))}{2^n}\leq\frac{\ell(T_{2^n}(\phi,D),T_{2^n}(\phi,\bar H_0))}{2^n}+\frac{\ell(T_{2^n}(\phi,\bar H_0))}{2^n}\leq h(\bar\phi_{G/H})+h(\phi\rest_H)+2\varepsilon\,.
%\end{split}
\end{equation*}
This holds for every $\varepsilon>0$ and every $D\in\mathcal{F}_C(G)$, therefore we have the thesis.
\end{proof}

%\begin{coro}
%Let $G$ be a locally finite and normal group, $\phi\in \End(G)$, $H$ a
%$\phi$-invariant normal subgroup of G and $\bar \phi\in \End(G/H)$ the endomorphism induced by $\phi$. Then
%$$ h(\phi)=h(\phi\rest_H)+h(\bar \phi)\,.$$
%\end{coro}
%\begin{proof}
%Follows by Lemma~\ref{lfn} and Theorem~\ref{atq}.
%\end{proof}
%
%
%\begin{coro}
%Let $G$ be a locally finite quasihamiltonian group, $\phi\in \End(G)$, $H$ a
%$\phi$-invariant normal subgroup of G and $\bar \phi\in \End(G/H)$ the endomorphism induced by $\phi$. Then
%$$ h(\phi)=h(\phi\rest_H)+h(\bar \phi)\,.$$
%\end{coro}
%\begin{proof}
%Follows by Remark~\ref{qht} and Theorem~\ref{atq}.
%\end{proof}

\end{document}